\documentclass[12pt, a4paper]{amsart}
\usepackage{amsmath}
\usepackage{geometry,amsthm,graphics,tabularx,amssymb,shapepar}
\usepackage{amscd}
\usepackage[all,2cell,dvips]{xy}

\newcommand{\BC}{{\mathbb {C}}}

\newcommand{\CC}{{\mathcal {C}}}
\newcommand{\CDH}{{\mathcal {D}\mathcal {H}}}

\newcommand{\CFH}{{\mathcal {F}\mathcal {H}}}

\newcommand{\CH}{{\mathcal {H}}}

\newcommand{\CS}{{\mathcal {S}}}

\newcommand{\C}{{\mathbb {C}}}

\newcommand{\RD}{{\mathrm {D}}}

\newcommand{\RU}{{\mathrm {U}}}

\newcommand{\RZ}{{\mathrm {Z}}}

\newcommand{\Ad}{{\mathrm{Ad}}}

\newcommand{\con}{{\mathrm{C}}}

\newcommand{\GL}{{\mathrm{GL}}}

\newcommand{\Hom}{{\mathrm{Hom}}}
\renewcommand{\Im}{{\mathrm{Im}}}

\newcommand{\Ker}{{\mathrm{Ker}}}

\newcommand{\tr}{{\mathrm{tr}}}

\newcommand{\vsp}{{\vspace{0.2in}}}
\newcommand{\vvsp}{{\vspace{0.1in}}}

\newcommand{\g}{\mathfrak g}
\newcommand{\gC}{{\mathfrak g}_{\C}}

\newcommand{\ve}{{\vee}}

\newcommand{\abs}[1]{\lvert#1\rvert}

\newcommand{\la}{\langle}
\newcommand{\ra}{\rangle}

\newcommand{\be}{\begin {equation}}
\newcommand{\ee}{\end {equation}}
\newcommand{\bee}{\begin {equation*}}
\newcommand{\eee}{\end {equation*}}

\renewcommand{\mid}{\,:\,}

\theoremstyle{plain}
\newtheorem{thm}{Theorem}[section]
\newtheorem{cor}[thm]{Corollary}
\newtheorem{lem}[thm]{Lemma}

\newtheorem{rmk}[thm]{Remark}

\begin{document}

\title[Gelfand-Kazhdan criterion]{A general form of Gelfand-Kazhdan criterion}

\author [B. Sun] {Binyong Sun}
\address{Academy of Mathematics and Systems Science\\
Chinese Academy of Sciences\\
Beijing, 100190, P.R. China} \email{sun@math.ac.cn}

\author [C.-B. Zhu] {Chen-Bo Zhu}
\address{Department of Mathematics\\
National University of Singapore\\
Block S17, 10 Lower Kent Ridge Road\\
Singapore 119076} \email{matzhucb@nus.edu.sg}

\thanks{}

\begin{abstract}
We formalize the notion of matrix coefficients for distributional
vectors in a representation of a real reductive group, which consist
of generalized functions on the group. As an application, we state
and prove a Gelfand-Kazhdan criterion for a real reductive group in
very general settings.
\end{abstract}

\maketitle

\section{Tempered generalized functions and Casselman-Wallach representations}

In this section, we review some basic terminologies in
representation theory, which are necessary for this article. The two
main ones are tempered generalized functions and Casselman-Wallach representations. We refer the readers to \cite{W1,W2} as general references.

Let $G$ be a real reductive Lie group, by which we mean that
\begin{itemize}
        \item[(a)]
            the Lie algebra $\g$ of $G$ is reductive;
        \item[(b)]
            $G$ has finitely many connected components; and
        \item[(c)]
            the connected Lie subgroup of $G$ with Lie algebra
            $[\g,\g]$ has a finite center.
\end{itemize}

We say that a (complex valued) function $f$ on $G$ is of moderate
growth if there is a continuous group homomorphism
\[
  \rho:\ G\rightarrow \GL_n(\BC), \ \text{for some $n\geq 1$},
\]
such that
\[
   \abs{f(x)}\leq \tr (\overline{\rho(x)}^{\,t} \rho(x))+\tr (\overline{\rho(x^{-1})}^{\,t} \rho(x^{-1})),\quad x\in
   G.
\]
Here ``$\bar{\phantom{p}}$" stands for the complex conjugation, and
``$\phantom{p}^t$" the transpose, of a matrix. A smooth function
$f\in \con^\infty(G)$ is said to be tempered if $Xf$ has moderate
growth for all $X$ in the universal enveloping algebra
$\RU(\g_\BC)$. Here and as usual, $\g_\BC$ is the complexification
of $\g$, and $\RU(\g_\BC)$ is identified with the space of all left
invariant differential operators on $G$. Denote by $\con^{\,\xi}(G)$
the space of all tempered functions on $G$.

A smooth function $f\in \con^\infty(G)$ is called Schwartz if
\[
  \abs{f}_{X,\phi}:=\mathrm{sup}_{x\in G} \,\phi(x)\,\abs{(Xf)(x)}<
  \infty
\]
for all $X\in \RU(\g_\BC)$, and all positive functions $\phi$ on $G$
of moderate growth. Denote by $\con^{\,\varsigma}(G)$ the space of
Schwartz functions on $G$, which is a nuclear Fr\'{e}chet space under the
seminorms $\{\abs{\,\cdot}_{X,\phi}\}$. (See \cite{Tr} or \cite{Ta} for the notion as well as
basic properties of nuclear Fr\'{e}chet spaces.) We define the nuclear Fr\'{e}chet
space $\RD^{\,\varsigma}(G)$ of Schwartz densities on $G$ similarly.
Fix a Haar measure $dg$ on $G$, then the map
\[
  \begin{array}{rcl}
               \con^{\,\varsigma}(G)&\rightarrow& \RD^{\,\varsigma}(G),\\
                                    f&\mapsto &f\,dg
  \end{array}
\]
is a topological linear isomorphism. We define a tempered
generalized function on $G$ to be a continuous linear functional on
$\RD^{\,\varsigma}(G)$. Denote by $\con^{-\xi}(G)$ the space of all
tempered generalized functions on $G$, equipped with the strong dual topology. This topology coincides with the topology of uniform convergence on compact subsets of $\RD^{\,\varsigma}(G)$, due to the fact that every bounded subset of a complete nuclear space is relatively compact.  Note that
$\con^{\,\xi}(G)$ is canonically identified with a dense subspace of
$\con^{-\xi}(G)$:
\[\con^{\,\xi}(G)\hookrightarrow \con^{-\xi}(G).\]

\begin{rmk} In \cite{W2}, the space $\con^{\,\varsigma}(G)$ is denoted by
$\CS (G)$ and is called the space of rapidly decreasing functions on $G$. Note that $\con^{\,\varsigma}(G)$ (or $\CS (G)$) is different from Harish-Chandra's Schwartz space of $G$, which is traditionally denoted by $\CC (G)$.
\end{rmk}

\vsp

By a representation of $G$, or just a representation when $G$ is
understood, we mean a continuous linear action of $G$ on a complete,
locally convex, Hausdorff, complex topological vector space. When no
confusion is possible, we do not distinguish a representation with
its underlying space. Let $V$ be a representation. It is said to be
smooth if the action map $G\times V\rightarrow V$ is smooth as a map
of infinite dimensional manifolds. Recall that in general, if $E$ and $F$ are two complete,
locally convex, Hausdorff, real topological vector spaces, and $U_E$ and $U_F$ are open subsets of $E$ and $F$ respectively, a continuous ($\con^0$) map $f:U_E\rightarrow U_F$ is said to be $\con^1$ if the differential
\[
  \begin{array}{rcl}
  df:U_E\times E&\rightarrow& F, \\
   x,v&\mapsto& \lim_{t\rightarrow0}\frac{f(x+tv)-f(x)}{t}
   \end{array}
\]
exists and is $\con^0$. Inductively $f$ is said to be $\con^k$ ($k\geq 2$) if it is $\con^1$ and $df$ is $\con^{k-1}$. We say that $f$ is smooth if it is $\con^k$ for all $k\geq 0$. With this notion of smoothness, we define (smooth) manifolds and smooth maps between them as in the finite dimension case.  See \cite{GN}, for example, for more details.

Denote by $\con(G;V)$ the space of $V$-valued continuous functions
on $G$. It is a complete locally convex space under the topology of
uniform convergence on compact sets. Similarly, denote by
$\con^\infty(G;V)$ the (complete locally convex) space of smooth
$V$-valued functions, with the usual smooth topology (which is defined by the seminorms $\abs{f}_{\Omega,D,\mu}:=\mathrm{sup}_{x\in \Omega} \abs{Df(x)}_\mu$, where $\Omega$ is a compact subset of $G$, $D$ is a differential operator on $G$, and $\abs{\,\cdot\,}_\mu$ is a continuous seminorm on $V$).

Define
\[
  V(\infty):=\{v\in V\mid  c_v\in \con^\infty(G;V)\},
\]
where $c_v\in \con(G;V)$ is given by $c_v(g):=gv$. This is a $G$-stable subspace of $V$.  It is easy to check that the linear map
\begin{equation}\label{orbitmap}
    V(\infty)\rightarrow \con^\infty(G;V),\quad   v\mapsto c_v
\end{equation}
is injective and has closed image. Identify $V(\infty)$ with the image of \eqref{orbitmap}, and equip on it the subspace topology of $\con^\infty(G;V)$. Then $V(\infty)$ becomes a smooth
representation of $G$, which is called the smoothing of $V$. The inclusion map $V(\infty)\rightarrow V$ is continuous since it can be identified with the map of evaluating at the identity $1\in G$. If $V$
is smooth, then this inclusion map is a homeomorphism, and hence $V(\infty)=V$ as a representation of $G$. In this
case, its differential is defined to be the continuous $\RU(\g_\BC)$
action given by
\[
  X v= (X c_v)(1), \quad X\in \RU(\g_\BC), \ v\in V.
\]

The representation $V$ is said to be $\mathrm{Z}(\g_\BC)$ finite if
a finite codimensional ideal of $\RZ(\g_\BC)$ annihilates
$V(\infty)$, where $\RZ(\g_\BC)$ is the center of $\RU(\g_\BC)$. It
is said to be admissible if every irreducible representation of a
maximal compact subgroup $K$ of $G$ has finite multiplicity in $V$.
A representation of $G$ which is both admissible and
$\mathrm{Z}(\g_\BC)$ finite is called a Harish-Chandra
representation.

The representation $V$ is said to be of moderate growth if for every
continuous seminorm $\abs{\phantom{f}}_\mu$ on $V$, there is a
positive function $\phi$ on $G$ of moderate growth, and a continuous
seminorm $\abs{\phantom{f}}_\nu$ on $V$ such that
\[
   \abs{gv}_\mu\leq \phi(g)\abs{v}_\nu,\quad \textrm{for all } g\in
   G, \  v\in V.
\]

The representation $V$ is called a Casselman-Wallach representation if the
space $V$ is Fr\'{e}chet, and the representation is smooth and of
moderate growth, and Harish-Chandra. Following Wallach (\cite{W2}), the category of all such $V$ is
denoted by $\CFH $ (the morphisms being
$G$-intertwining continuous linear maps). The strong dual of a
Casselman-Wallach representation is again a representation which
is smooth and Harish-Chandra. Representations which are isomorphic
to such strong duals form a category, which is denoted by $\CDH$. By the
Casselman-Wallach globalization theorem, both the category $\CFH$
and $\CDH$ are equivalent to the category $\CH$ of admissible
finitely generated $(\g_\C, K)$-modules  (\cite{Cass}, \cite[Chapter
11]{W2}).


From the theory of real Jacquet modules (by Casselman and Wallach), and the Casselman-Wallach globalization theorem, every Casselman-Wallach representation is the smoothing of a Hilbert representation. In addition, all representation spaces in $\CFH$ are automatically nuclear Fr\'{e}chet (and in particular are reflexive), and all morphisms in $\CFH$ and $\CDH$ are automatically topological homomorphisms with closed image. See \cite[Chapter 11]{W2}. Recall that in general, a linear map $\lambda: E\rightarrow F$ of topological vector spaces is called a topological homomorphism if the induced linear isomorphism  $E/\Ker(\lambda)\rightarrow \Im (\lambda)$ is a topological isomorphism, where $E/\Ker(\lambda)$ is equipped with the quotient topology of $E$, and the image $\Im(\lambda)$ is equipped with the subspace topology of $F$.

\section{Statement of results}

Let $U^\infty, V^\infty$ be a pair of Casselman-Wallach representations of $G$ which are contragredient to each other, i.e.,
we are given a $G$-invariant nondegenerate continuous bilinear map
\begin{equation}\label{pv1v2}
   \la\,,\,\ra: U^\infty\times V^\infty\rightarrow \BC.
\end{equation}
Note that $U^\infty$ is the only Casselman-Wallach representation whose underlying $(\g_\C,K)$-module is the contragredient of that of $V^\infty$, and vice versa.

Denote by $U^{-\infty}$ the strong dual of $V^\infty$. This is the only representation in $\CDH$ which has the same underlying $(\g_\C,K)$-module as that of $U^\infty$. Similarly, denote by $V^{-\infty}$ the strong
dual of $U^\infty$. For any $u\in U^{\infty}$, $v\in V^{\infty}$,
the (usual) matrix coefficient $c_{u\otimes v}$ is defined by
\begin{equation}
\label{dmc}
  c_{u\otimes v}(g):=\la gu,v\ra, \quad g\in G.
\end{equation}
By the moderate growth conditions of $U^\infty$ and $V^\infty$, one
easily checks that a matrix coefficient $c_{u\otimes v}$ is a tempered function on
$G$.

The following theorem, which defines the notion of matrix
coefficients for distributional vectors, is in a sense well-known.
See the work of Shilika (\cite[Section 3]{Sh}) in the context of unitary representations, and the work of
Kostant (\cite[Section 6.1]{Ko}) or Yamashita (\cite[Section 2.3]{Ya}) in the context of Hilbert representations.
With the benefit of the Casselman-Wallach theorem, it is of interest and most natural to state the result in the context of Casselman-Wallach representations. This is also partly justified by the increasing use of these representations, due to the recent progress in restriction problems for classical groups. One purpose of this note is to provide a detailed proof of this result.

\begin{thm}\label{thm1} Let $G$ be a real reductive group. Denote by $\con^{\xi}(G)$ (resp. $\con^{-\xi}(G)$) the space of all
tempered functions (resp. tempered generalized functions) on $G$. Let $(U^\infty, V^\infty)$ be a pair of Casselman-Wallach representations of $G$ which are
contragredient to each other. Then the matrix coefficient map
\begin{equation}\label{umatrix}
 \begin{array}{rcl}
 U^\infty\times V^\infty&\rightarrow& \con^{\,\xi}(G),\\
           (u,v)&\mapsto & c_{u\otimes v}
 \end{array}
\end{equation}
extends to a continuous bilinear map
\[
  U^{-\infty}\times V^{-\infty}\rightarrow \con^{-\xi}(G),
\]
and the induced $G\times G$ intertwining continuous linear map
\begin{equation}\label{mapc}
   c \,:\, U^{-\infty}\widehat \otimes V^{-\infty}\rightarrow \con^{-\xi}(G)
\end{equation}
is a topological homomorphism with closed image.
\end{thm}

Here ``$\widehat \otimes$" stands for the completed projective
tensor product of Hausdorff locally convex topological vector
spaces. In our case, this coincides with the completed epsilon
tensor product as the spaces involved are nuclear. Recall again that
a linear map $\lambda: E\rightarrow F$ of topological vector spaces
is called a topological homomorphism if the induced linear
isomorphism $E/\Ker(\lambda)\rightarrow \Im (\lambda)$ is a
topological isomorphism, where $E/\Ker(\lambda)$ is equipped with
the quotient topology of $E$, and the image $\Im(\lambda)$ is
equipped with the subspace topology of $F$. The action of $G\times
G$ on $\con^{-\xi}(G)$ is obtained by continuously extending its
action on $\con^{\xi}(G)$:
\[
  ((g_1,g_2)f)(x):=f(g_2^{-1}xg_1).
\]

\begin{rmk}\label{rch}
(A) Denote by $t_{U_\infty}$ the paring $V^\infty\times U^\infty\rightarrow \C$, and view it as an element of $U^{-\infty}\widehat \otimes V^{-\infty}$. Then $c(t_{U_\infty})\in \con^{-\xi}(G)$ is the character of the representation $U^\infty$.

\vvsp

\noindent (B) Let $U_1^\infty, U_2^\infty, \cdots, U_k^\infty$ be pairwise inequivalent irreducible Casselman-Wallach representations of $G$. Let $V_i^\infty$ be a Casselman-Wallach  representation of $G$ which is contragredient to $U_i^\infty$, $i=1,2,\cdots, k$. Then the second assertion of Theorem \ref{thm1} implies that the sum
\[
  \bigoplus_{i=1}^k U_i^{-\infty}\widehat \otimes V_i^{-\infty} \rightarrow \con^{-\xi}(G)
\]
of the matrix coefficient maps is a topological embedding with closed image.
\end{rmk}

\vsp A second purpose of this note (and additional reason for writing
down a proof of Theorem \ref{thm1}) is to prove the following
generalized form of the Gelfand-Kazhdan criterion. For applications
towards uniqueness of certain degenerate Whittaker models, it is
highly desirable (and in fact necessary) to have the most general form of the
Gelfand-Kazhdan criterion. We refer the reader to \cite{JSZ} for one
such application.

\begin{thm}\label{gelfand} Let $S_1$ and $S_2$ be two closed subgroups of $G$,
with continuous characters
\[
  \chi_i: S_i\rightarrow \BC^\times,\quad i=1,2.
\]
\begin{itemize}
\item[(a)]
Assume that there is a continuous anti-automorphism $\sigma$ of $G$
such that for every $f\in \con^{-\xi}(G)$ which is an eigenvector of
$\operatorname{U}(\gC)^G$, the conditions
\[
  f(sx)=\chi_{1}(s)f(x), \quad s\in S_1,
\]
and
\[
  f(xs)=\chi_{2}(s)^{-1}f(x), \quad s\in S_2
\]
imply that
\[
   f(x^\sigma)=f(x).
\]
Then for any pair of irreducible Casselman-Wallach representations $(U^\infty, V^\infty)$ of $G$ which are contragredient to
each other, one has that
\begin{equation*}\label{dhom}
  \dim \Hom_{S_1}(U^\infty, \C_{\chi_{1}}) \, \dim \Hom_{S_2}
  (V^\infty,\C_{\chi_{2}})\leq 1.
\end{equation*}
\item[(b)]
Assume that for every $f\in \con^{-\xi}(G)$ which is an eigenvector
of $\operatorname{U}(\g)^G$, the conditions
\[
  f(sx)=\chi_{1}(s)f(x), \quad s\in S_1,
\]
and
\[
  f(xs)=\chi_{2}(s)^{-1}f(x), \quad s\in S_2
\]
imply that
\[
   f=0.
\]
Then for any pair of irreducible Casselman-Wallach representations $(U^\infty, V^\infty)$ of $G$ which are contragredient to
each other, one has that
\begin{equation*}
  \dim \Hom_{S_1}(U^\infty, \C_{\chi_{1}}) \,
  \dim \Hom_{S_2}(V^\infty,\C_{\chi_{2}})=0.
\end{equation*}
\end{itemize}
\end{thm}

Here and as usual, $\operatorname{U}(\g_\BC)^G$ is identified with
the space of bi-invariant differential operators on $G$,
$\C_{\chi_{i}}$ is the one dimensional representation of $S_i$ given
by the character $\chi_i$, and ``$\Hom_{S_i}$" stands for continuous
$S_i$ homomorphisms. The equalities in the theorem are to be
understood as equalities of generalized functions. For example,
$f(sx)$ denotes the left translation of $f$ by $s^{-1}$. Similar
notations apply throughout this article.

\begin{rmk}
The original Gelfand-Kazhdan criterion is in \cite{GK} (for the non-archimedean case), and their idea has been very influential ever since. Various versions for real reductive groups have appeared in the literature, including \cite{Sh} and \cite{Ko} (for the study of Whittaker models, but both implicitly). Later works which state some versions of Gelfand-Kazhdan criterion explicitly include that of H. Yamashita (\cite[Theorem 2.10]{Ya}), and of A. Aizenbud, D. Gourevitch and E. Sayag (\cite[Section 2]{AGS}). The papers \cite{Sh} and \cite{Ya} have been particularly instructive for the current article.
\end{rmk}

As a consequence of Part (a) of Theorem \ref{gelfand}, we have the following criterion of a strong Gelfand pair.
\begin{cor}\label{sgel}
Let $G'$ be a reductive closed subgroup of the real reductive group $G$. Let $\sigma$ be
a continuous anti-automorphism of $G$ such that $\sigma(G')=G'$.
Assume that for every $f\in \con^{-\xi}(G)$, the
condition
\[
   f(gxg^{-1})=f(x)\quad \textrm{ for all } g\in G'
\]
implies that
\[
  f(x^\sigma)=f(x).
\]
Then for all irreducible Casselman-Wallach representation $V$ of
$G$, and $V'$ of $G'$, the space of $G'$-invariant continuous bilinear functionals on $V\times V'$ is at most one dimensional.
\end{cor}

\section{Proof of Theorem \ref{thm1}}
Let $(U, \la\,,\,\ra_U)$ be a Hilbert space which carries a
continuous representation of $G$ so that its smoothing coincides
with $U^\infty$.
Denote by $V$ the strong dual of $U$, which carries a representation
of $G$. (Its topology is given by the inner product
\[
  \la \bar{u_1},\bar{u_2}\ra_V:=\la u_2,u_1\ra_U, \quad u_1,u_2\in
  U,
\]
where $\bar{u_i}\in V$ is the linear functional $\la \,\cdot\,,
u_i\ra_U$ on $U$.) Note that the smoothing of $V$ coincides with
$V^\infty$. Recall, as is well-known, that the three pairs $U$ and
$V$, $U^{\infty}$ and $V^{-\infty}$, and $U^{-\infty}$ and
$V^{\infty}$, are strong duals of each other as representations of
$G$. For $u\in U$, $v\in V$, set
\[
  \abs{u}_U:=\sqrt{\la u,u\ra_U} \quad\textrm{ and }\quad \abs{v}_V
  :=\sqrt{\la v,v\ra_V}.
\]

\begin{lem}\label{absc}
There is a continuous seminorm $\abs{\,\cdot}_G$ on
$\RD^{\,\varsigma}(G)$ such that
\[
  \int_G \abs{f(g)} \,\abs{gu}_U\,dg\leq
  \abs{\omega}_G\,\abs{u}_U,\quad \omega=f\,dg\in
  \RD^{\,\varsigma}(G),\,
  u\in U.
\]
\end{lem}
\begin{proof}
This is well known, and follows easily from the facts that
\begin{itemize}
 \item[(a)]
  $U$ is (automatically) of moderate growth (\cite[Lemma 2.A.2.2]{W1}), and
 \item[(b)]
  there is a positive continuous function $\phi$ on $G$ of moderate growth
  so that $1/\phi$ is integrable. See \cite[Lemma 2.A.2.4]{W1}.
\end{itemize}
\end{proof}

By Lemma \ref{absc}, for any $\omega\in \RD^{\,\varsigma}(G)$ and
$u\in U$, the integral (in the sense of Riemann)
\begin{equation}\label{omegau}
  \omega u:=\int_G \omega(g) \,gu
\end{equation}
converges absolutely, and thus defines a vector in $U$. Furthermore,
the bilinear map
\begin{equation}\label{b1}
    \RD^{\,\varsigma}(G)\times U\rightarrow U,\quad(\omega, u)\mapsto \omega u
 \end{equation}
is continuous.

\begin{lem}\label{smooth}
For $\omega\in \RD^{\,\varsigma}(G)$ and $u\in U$, we have $\omega
u\in U^{\infty}$.
\end{lem}
\begin{proof}
Denote by $L$ the representation of $G$ on $\RD^{\,\varsigma}(G)$ by
left translations. Thus for $g\in G$ and $\omega\in
\RD^{\,\varsigma}(G)$, $L_g(\omega)$ is the push forward of $\omega$
via the map
\[
     G\rightarrow G,\quad x\mapsto gx.
\]
It is routine to check that
\[
  L: \ G\times \RD^{\,\varsigma}(G)\rightarrow \RD^{\,\varsigma}(G)
\]
is a smooth representation. For $X\in \RU(\g_\BC)$, denote by
\[
  L_X:\ \RD^{\,\varsigma}(G)\rightarrow \RD^{\,\varsigma}(G)
\]
its differential. Trivially we have
\begin{equation}\label{comega}
  c_{\omega u}(g)=(L_g(\omega))u,\quad g\in G, u\in U.
\end{equation}
This implies that $c_{\omega u}\in \con^{\infty}(G;U)$, namely
$\omega u\in U^\infty$.
\end{proof}

The following two lemmas are refinements of \cite[Proposition
3.2]{Sh}.

\begin{lem}\label{smoothing1}
The bilinear map
\[
  \begin{array}{rcl}
   \Phi _U: \ \ \RD^{\,\varsigma}(G)\times U&\rightarrow &U^\infty,\\
  (\omega, u)&\mapsto &\omega u.
  \end{array}
\]
is continuous.
\end{lem}
\begin{proof}
By the defining topology on $U^\infty$, we need to show that the map
\[
   \RD^{\,\varsigma}(G)\times U\rightarrow \con^\infty(G;U),\quad(\omega, u)\mapsto c_{\omega u}.
\]
is continuous. In view of the topology on $\con^\infty(G;U)$, this
is equivalent to showing that the bilinear map
\[
      \RD^{\,\varsigma}(G)\times U\rightarrow \con(G;U),\quad(\omega, u)\mapsto X(c_{\omega u}),
\]
is continuous for all $X\in \RU(\g_\BC)$. This is clearly true by
observing that
\[
  X(c_{\omega u})=c_{(L_X(\omega))u}, \quad \omega\in
  \RD^{\,\varsigma}(G), \ u\in U.
\]
\end{proof}

For any $\omega\in  \RD^{\,\varsigma}(G)$, denote by $\omega^\vee$
its push forward via the map
\[
      G \rightarrow  G,\quad g\mapsto g^{-1}.
\]
Applying Lemma \ref{smoothing1} to $V$, we get a continuous bilinear
map
\[
     \Phi _V: \ \ \RD^{\,\varsigma}(G)\times V \rightarrow V^\infty,\quad (\omega, v)\mapsto \omega v.
\]
Now for any $\omega\in  \RD^{\,\varsigma}(G)$, we define the
continuous linear map
\[
   U^{-\infty}\rightarrow U,\quad u\mapsto \omega u
\]
to be the transpose of
\[
    V \rightarrow V^\infty,\quad
   v\mapsto \omega^\vee\, v,
\]
i.e.,
\begin{equation}\label{omegauv}
\la \omega u,v\ra = \la u, \omega ^\vee v\ra , \quad u\in
U^{-\infty}, \ v\in V.
\end{equation}

\begin{lem}
The bilinear map
\begin{equation}\label{bm}
   \begin{array}{rcl}
   \Phi _V^{\vee}: \ \ \RD^{\,\varsigma}(G)\times U^{-\infty}& \rightarrow &U,\\
  (\omega, u)&\mapsto &\omega u
  \end{array}
\end{equation}
is separately continuous and extends (\ref{b1}).
\end{lem}
\begin{proof}
It is routine to check that (\ref{bm}) extends (\ref{b1}). We
already know that (\ref{bm}) is continuous in the second variable.

Fix $u\in U^{-\infty}$, then the continuity of the bilinear map
\[
    \theta_u: \RD^{\,\varsigma}(G)\times V \rightarrow \BC,\quad(\omega, v)\mapsto \la u, \omega^\vee v\ra.
\]
clearly implies the continuity of the map
\[
   \RD^{\,\varsigma}(G) \rightarrow U,\quad \omega  \mapsto \omega u=\theta_u(\omega, \,\cdot\,).
\]
\end{proof}

\begin{lem} The image of $\Phi _V^{\vee}$ is contained in $U^\infty$, and
the induced bilinear map
\begin{equation}\label{b2}
   \begin{array}{rcl}
   \Phi _V^{\vee}:\ \ \RD^{\,\varsigma}(G)\times U^{-\infty}& \rightarrow &U^\infty,\\
  (\omega, u)&\mapsto &\omega u
  \end{array}
\end{equation}
is separately continuous.
\end{lem}
\begin{proof}
By chasing the definition of $\omega u$, we see that the equality
(\ref{comega}) still holds for all $\omega\in \RD^{\,\varsigma}(G)$
and $u\in U^{-\infty}$. Again, this implies that $\omega u\in
U^\infty$.

The proof for the separate continuity of $\Phi _V^{\vee}$ is similar to that of Lemma
\ref{smoothing1}. We need to prove that the map
\[
   \RD^{\,\varsigma}(G)\times U^{-\infty}\rightarrow \con^\infty(G;U),\quad
  (\omega, u) \mapsto c_{\omega u}.
\]
is separately continuous. This is the same as that the bilinear map
\[
     \RD^{\,\varsigma}(G)\times U^{-\infty}\rightarrow \con(G;U),\quad
  (\omega, u)\mapsto X(c_{\omega u}),
\]
is separately continuous for all $X\in \RU(\g_\BC)$. This is again
true by checking that
\[
  X(c_{\omega u})=c_{(L_X(\omega))u}, \quad \omega\in
  \RD^{\,\varsigma}(G),\
  u\in U^{-\infty}.
\]

\end{proof}

To summarize, we get a separately continuous bilinear map
\[
  \RD^{\,\varsigma}(G)\times U^{-\infty} \rightarrow U^\infty,\quad (\omega, u)\mapsto \omega
  u,
\]
which extends the action map
\[
  \RD^{\,\varsigma}(G)\times U \rightarrow U^\infty,\quad (\omega, u)\mapsto \omega
  u.
\]
Similarly, we have a separately continuous bilinear map
\[
  \RD^{\,\varsigma}(G)\times V^{-\infty} \rightarrow V^\infty,\quad (\omega, v)\mapsto \omega
  v,
\]
which extends the action map
\[
  \RD^{\,\varsigma}(G)\times V \rightarrow V^\infty,\quad (\omega, v)\mapsto \omega
  v.
\]

Now define the (distributional) matrix coefficient map by
\begin{equation}\label{mf}
  \begin{array}{rcl}
   &&c: \, U^{-\infty}\times V^{-\infty}\rightarrow \con^{-\xi}(G),\\
   &&c_{u\otimes v}(\omega):=\la \omega u, v\ra=\la u, \omega ^\vee v\ra, \ \ \omega \in
   \RD^{\,\varsigma}(G).
  \end{array}
\end{equation}
The last equality is implied by \eqref{omegauv} and the afore-mentioned
separate continuity statements.

\begin{lem}
The matrix coefficient map $c$ defined in (\ref{mf}) is continuous.
\end{lem}
\begin{proof}
Note that the spaces $U^{-\infty}$, $V^{-\infty}$ and
$\con^{-\xi}(G)$ are all strong duals of reflexive Fr\'{e}chet
spaces. Therefore by \cite[Theorem 41.1]{Tr}, it suffices to show
that the map (\ref{mf}) is separately continuous.  First fix $u\in
U^{-\infty}$, then the map
\[
     V^{-\infty}\rightarrow \con^{-\xi}(G),\quad  v\mapsto c_{u\otimes v},
 \]
is continuous since it is the transpose of the continuous linear map
\[
     \RD^{\,\varsigma}(G)\rightarrow  U^\infty,\quad
     \omega \mapsto  \omega u.
  \]
Similarly, fix $v\in V^{-\infty}$, the map
\[
      U^{-\infty}\rightarrow \con^{-\xi}(G),\quad
     u \mapsto c_{u\otimes v},
 \]
is continuous since it is the transpose of the continuous linear
map
\[
     \RD^{\,\varsigma}(G)\rightarrow  V^\infty,\quad
     \omega \mapsto  \omega^\ve v.
  \]

\end{proof}

It is straightforward to check that (\ref{mf}) extends the usual
matrix coefficient map (\ref{dmc}). The proof of the first assertion of Theorem \ref{thm1}
is now complete.

\vsp

To prove the second assertion of Theorem \ref{thm1} (the generalized matrix coefficient map \eqref{mapc} is a topological homomorphism with closed image), we need two elementary lemmas.

\begin{lem}\label{top1}
Let $\lambda:E\rightarrow F$ be a $G$ intertwining continuous linear map of representations of $G$. Assume that $E$ is Fr\'{e}chet, smooth, and of moderate growth, and $F$ is a Casselman-Wallach representation.  Then $E/\Ker(\lambda)$ is a Casselman-Wallach representation, and $\lambda$ is a topological homomorphism.
\end{lem}
\begin{proof}
The quotient representation $E/\Ker(\lambda)$ is clearly  Fr\'{e}chet, smooth, and of moderate growth. It is $\mathrm{Z}(\g_\BC)$ finite and $K$ finite since it is  mapped injectively to $F$.  Therefore it is a Casselman-Wallach representation. The second assertion is a consequence of Casselman-Wallach globalization Theorem.

\end{proof}

\begin{lem}\label{top2}
Let $\lambda:E\rightarrow F$ be a continuous linear map of nuclear Fr\'{e}chet spaces. Equip the dual spaces $E'$ and $F'$ with the strong dual topologies. Then $\lambda$ is a topological homomorphism if and only if its transpose $\lambda^t: F'\rightarrow E'$ is. When this is the case, both $\lambda$ and $\lambda^t$ have closed images.
\end{lem}
\begin{proof}
The first assertion is a special case of \cite[Section IV.2, Theorem 1]{Bo}. (Recall that every bounded set in a complete nuclear space is relatively compact.)

Now assume that $\lambda$ is a topological homomorphism. Then as a Hausdorff quotient of a Fr\'{e}chet space, $E/\Ker(\lambda)$ is complete, and so is $\Im(\lambda)$, which implies that $\Im(\lambda)$ is closed in $F$.
By the Extension Theorem of continuous linear functionals, the image of $\lambda^t$ consists of all elements in $E'$ which vanish on $\Ker(\lambda)$. This is closed in $E'$.
\end{proof}

Recall that both $V^\infty\widehat \otimes U^\infty$ and $\RD^{\,\varsigma}(G)$ are nuclear Fr\'{e}chet spaces. In particular, they are both reflexive. The map $c$ of \eqref{mapc} is the transpose of a $G\times G$ intertwining continuous linear map
\[
  c^t: \RD^{\,\varsigma}(G)\rightarrow V^\infty\widehat{\otimes} U^\infty.
\]
(Here we have used the canonical isomorphism $E'\widehat{\otimes}F'\simeq (E\widehat{\otimes}F)'$, for nuclear Fr\'{e}chet spaces $E$ and $F$. See
\cite[Proposition 50.7]{Tr}.)
Lemma \ref{top1} for the group $G\times G$ implies that $c^t$ is a topological homomorphism. Lemma \ref{top2} then implies that $c$ is a topological homomorphism with closed image. This completes the proof of the second assertion of Theorem \ref{thm1}.

\section{Proof of Theorem \ref{gelfand}}

The argument is standard (cf. \cite{GK} or \cite{Sh}). We use the notation and
the assumption of Theorem \ref{gelfand}. As before, $U^{-\infty}$ is
the strong dual of $V^\infty$, and $V^{-\infty}$ is the strong dual
of $U^\infty$. Suppose that both
$\Hom_{S_1}(U^\infty,\C_{\chi_{1}})$ and
$\Hom_{S_2}(V^\infty,\C_{\chi_{2}})$ are non-zero. Pick
\[
  0\ne u_0\in \Hom_{S_2}(V^\infty,\C_{\chi_{2}})\subset U^{-\infty}
\]
and
\[
  0\ne v_0\in \Hom_{S_1}(U^\infty,\C_{\chi_{1}})\subset V^{-\infty}.
\]
Then the matrix coefficient $c_{u_0\otimes v_0}\in \con^{-\xi}(G)$
satisfies the followings:
\[
  \left\{
    \begin{array}{l}
       \textrm{it is an eigenvector of
       $\operatorname{U}(\g_\BC)^G$,}\, \, \,
       (\text{by the irreducibility hypothesis})\\
         \textrm{} c_{u_0\otimes v_0}(sx)=\chi_{1}(s)\,c_{u_0\otimes v_0}(x),
         \quad s\in S_1, \textrm{ and }\smallskip\\
              \textrm{}   c_{u_0\otimes v_0}(xs)=\chi_{2}(s)^{-1}
              \,c_{u_0\otimes v_0}(x), \quad s\in S_2.
  \end{array}
  \right.
\]
By the assumption of the theorem, we have
\begin{equation}\label{inv1}
   c_{u_0\otimes v_0}(x^\sigma)= c_{u_0\otimes v_0}(x).
\end{equation}

\begin{lem}\label{iff}
Let $\omega\in \RD^{\,\varsigma}(G)$. Denote by
\[
  \sigma_* : \RD^{\,\varsigma}(G)\rightarrow \RD^{\,\varsigma}(G)
\]
the push forward map by $\sigma$. Then
\[
   \omega u_0=0\quad\textrm{if and only if}\quad (\sigma_*(\omega))^\vee v_0=0.
\]
\end{lem}

\begin{proof}
As a consequence of (\ref{inv1}), we have
\begin{equation}\label{kernel}
    c_{u_0\otimes v_0}(\omega)=0\quad\textrm{if and only if}
    \quad c_{u_0\otimes v_0}(\sigma_*(\omega))=0.
\end{equation}

By the irreducibility of $U^\infty$, $\omega u_0=0$ if and only if
\[
   \la g(\omega u_0), v_0\ra=0\quad\textrm{for all } g\in G,
\]
i.e.,
\[
   \la (L_g\omega) u_0, v_0\ra=0\quad\textrm{for all } g\in G.
\]
By (\ref{kernel}), this is equivalent to saying that
\begin{equation*}\label{kernel2}
  \la (\sigma_*(L_g\omega) )u_0, v_0\ra=0\quad\textrm{for all } g\in G.
\end{equation*}
Now the lemma follows from the following elementary identity and the
irreducibility of $V^\infty$:
\begin{equation*}
  \phantom{=}\,\la (\sigma_*(L_g\omega) )u_0, v_0\ra
  =\la (\sigma_*\omega) (g^\sigma u_0), v_0\ra
  =\la  (g^\sigma u_0), (\sigma_*\omega)^\vee\, v_0\ra.
\end{equation*}
\end{proof}

\begin{proof}[End of proof of Theorem \ref{gelfand}]
Let
\[
0\ne u'_0\in \Hom_{S_2}(V^\infty,\C_{\chi_{2}})\subset U^{-\infty}
\]
be another element. Applying Lemma \ref{iff} twice, we get that for
all $\omega\in \RD^{\,\varsigma}(G)$,
\[
    \omega u_0=0\quad\textrm{if and only if}\quad \omega u_0'=0.
\]
Therefore the two continuous $G$ homomorphisms
\[
    \Phi: \omega \mapsto  \omega u_0,\quad \textrm{and}\quad
     \Phi': \omega \mapsto  \omega u'_0,
\]
from $\RD^{\,\varsigma}(G)$ to $U^\infty$ have the same kernel, say
$J$. Here and as before, we view $\RD^{\,\varsigma}(G)$ as a
representation of $G$ via left translations. Both $\Phi$ and $\Phi'$
induce nonzero $G$ homomorphisms into the irreducible Casselman-Wallach representation $U^{\infty}$
\[
  \bar{\Phi}, \bar{\Phi'}: \RD^{\,\varsigma}(G)/J\rightarrow  U^\infty,
\]
without kernel. Lemma \ref{top1} says that $\RD^{\,\varsigma}(G)/J$ is a Casselman-Wallach representation of $G$. By Schur's lemma for Casselman-Wallach representations, $\bar{\Phi'}$ is a scalar multiple of $\bar{\Phi}$, which implies that $u_0'$ is a scalar multiple of $u_0$. This proves that
\[
   \dim \Hom_{S_2}(V^\infty, \C_{\chi_{2}})=1.
\]
Similarly,
\[
   \dim \Hom_{S_1}(U^{\infty}, \C_{\chi_{1}})=1.
\]
This finishes the proof of Part (a) of Theorem \ref{gelfand}. Part (b) of Theorem \ref{gelfand} is immediate as the matrix coefficient $c_{u_0\otimes v_0}$ would have to be zero if there were
nonzero $u_0\in \Hom_{S_2}(V^\infty,\C_{\chi_{2}})$ and nonzero
$v_0\in \Hom_{S_1}(U^\infty,\C_{\chi_{1}})$.
\end{proof}

\vsp

\section{Proof of Corollary \ref{sgel}}

Let $G'$ be a reductive closed subgroup of the real reductive group $G$, and let $\sigma$ be a continuous anti-automorphism of $G$ such that $\sigma(G')=G'$, as in Corollary \ref{sgel}. Assume that for every $f\in \con^{-\xi}(G)$, the
condition
\[
   f(gxg^{-1})=f(x),\quad g\in G'
\]
implies that
\[
  f(x^\sigma)=f(x).
\]

Set
\[
  H:=G\times G',
\]
which contains $G$ as a subgroup. Denote by $S\subset H$ the group
$G'$ diagonally embedded in $H$. For any $x=(g,g')\in H$, set
\[
  x^\sigma:=(g^\sigma,{g'}^\sigma).
\]

\begin{lem}\label{invh}
If $f\in \con^{-\xi}(H)$ is invariant under the adjoint action, then it is $\sigma$-invariant.
\end{lem}
\begin{proof}
The assumption at the beginning of this section trivially implies that every invariant tempered generalized function on $G$ is $\sigma$-invariant. Since both $G$ and $G'$ are unimodular, it also implies that every invariant tempered generalized function on $G'$ is $\sigma$-invariant. The lemma follows easily from these two facts.
\end{proof}

\begin{lem}\label{tauinv}
If $f\in \con^{-\xi}(H)$ is a bi
$S$-invariant, then it is $\sigma$-invariant.
\end{lem}

\begin{proof}

The multiplication map
\[
  \begin{array}{rcl}
   m_H: S\times G\times S&\rightarrow &H\\
         (s_1, g, s_2)&\mapsto& s_1 g s_2
  \end{array}
\]
is a  surjective submersion. Let $f$ be a bi $S$-invariant
generalized function on $H$. Then its pull back has the form
\[
  m_H^*(f)=1\otimes f_G\otimes 1, \quad \textrm{with} \ f_G\in
  \con^{-\xi}(G).
\]

We use ``$\Ad$" to indicate the adjoint action. By considering the commutative diagram
\begin{equation*}
    \begin{CD}
        S\times G\times S@>>m_H>H\\
         @V\Ad_s\times \Ad_s\times \Ad_s VV           @V\Ad_sVV                \\
         S\times G\times S@>>m_H>H\,,
  \end{CD}
\end{equation*}
for all $s\in S$, we conclude that $f_G$ is invariant under the
adjoint action of $G'$. Therefore $f_G$ is $\sigma$-invariant by assumption.

Set
\[
  (s_1,g,s_2)^\sigma:=(s_2^\sigma, g^\sigma, s_1^\sigma),  \quad
  (s_1,g,s_2)\in S\times G\times S.
\]
Then $1\otimes f_G\otimes 1\in \con^{-\xi}( S\times G\times S)$
is also $\sigma$-invariant. We conclude that $f$ is
$\sigma$-invariant by appealing to the commutative diagram
\begin{equation*}
    \begin{CD}
        S\times G\times S@>>m_H>H\\
         @V\sigma VV           @V\sigma VV                \\
         S\times G\times S@>>m_H>H\,.
  \end{CD}
\end{equation*}

\end{proof}

Let $(V_H,
\rho)$ be an irreducible Casselman-Wallach representation of
$H$.
\begin{lem}\label{contrah}
Set
\[
  \rho_{-\sigma}(h):=\rho(h^{-\sigma}).
\]
Then $(V_H,\rho_{-\sigma})$ is an irreducible Casselman-Wallach representation of $H$ which is contragredient to $(V_H,\rho)$.
\end{lem}

\begin{proof}
Denote by
\[
  \chi_{\rho}\in \con^{-\xi}(H)
\]
the character of $(V_H, \rho)$. Then its contragredient
representation has character $\chi_{\rho}(h^{-1})$.

It is clear that $(V_H,\rho_{-\sigma})$ is an irreducible
 Casselman-Wallach representation, with character
$\chi_{\rho}(h^{-\sigma})$.  Since a character is always
invariant under the adjoint action, Lemma \ref{invh} implies that
\[
   \chi_{\rho}(h^{-1})=\chi_{\rho}(h^{-\sigma}).
\]
The lemma then follows from the well-known fact that an irreducible
Casselman-Wallach representation is determined by its
character.
\end{proof}

\begin{lem}
\label{sgelfand}
We have that
\[
  \dim \Hom_S (V_H,\C)\leq 1.
\]
\end{lem}

\begin{proof}
Denote by $U_H$ the irreducible Casselman-Wallach representation which is contragredient to $V_H$. Lemma \ref{tauinv} and Part (a) of Theorem \ref{gelfand} imply that
\[
 \dim \Hom_S(U_H,\C)\, \dim \Hom_S (V_H,\C)\leq 1.
\]
Lemma \ref{contrah} implies that
\[
  \dim \Hom_S (U_H,\C)=\dim \Hom_S (V_H,\C).
\]
We therefore conclude that $\dim \Hom_S (V_H,\C)\leq 1$.
\end{proof}

\vsp

We finish the proof of Corollary \ref{sgel} by taking $V_H=V\widehat \otimes V'$ in Lemma \ref{sgelfand}.

\vsp

\noindent Acknowledgements: the authors thank Bernhard Kr\"{o}tz for helpful discussions, and in particular
for pointing out an error in an early version of the manuscript. The comments of the referee also help
the authors to make several subtle points clearer. Binyong Sun is supported by NUS-MOE grant
R-146-000-102-112, and NSFC grants 10801126 and 10931006. Chen-Bo Zhu is supported by
NUS-MOE grant R-146-000-102-112.

\end{document}